\documentclass[10pt, twoside]{amsart}
\usepackage{epic,eepic,latexsym}
\boldmath

\usepackage{color}
\usepackage[usenames,dvipsnames,svgnames,table]{xcolor}
\usepackage[colorlinks=false,
            linkcolor=blue,
            urlcolor=blue,
            citecolor=blue]{hyperref}

 \newlength{\baseunit}               
 \newcount{\numlines}                
\usepackage[paperheight=240mm,paperwidth=170mm, left=22.0mm,right=22.0mm, top=22.0mm,bottom=22.0mm]{geometry}

\newtheorem{tm}{Theorem}[section]
\newtheorem{pr}[tm]{Proposition}

\newtheorem{co}[tm]{Corollary}
\newtheorem{definition}[tm]{Definition}

\usepackage{bigints}

\usepackage{adjustbox} 
\usepackage{mathrsfs}
\usepackage{enumerate}
\usepackage{mathtools}

\newcommand{\Sym}{\operatorname{Sym}}

\usepackage{tikz-cd}
\usepackage{mathrsfs}
\usepackage{amssymb}

\newcommand{\LL}{\mathscr{L}}

\newcommand{\PP}{\mathbb{P}}
\newcommand{\QQ}{\mathbb{Q}}

\newcommand{\CC}{\mathbb{C}}

\newcommand{\FF}{\mathscr{F}}

\newcommand{\OO}{\sO}

\newcommand{\II}{\mathscr{I}}

\newcommand{\sO}{\mathcal{O}}

\newcommand{\Mn}{\mathcal{M}}
\newcommand{\Hh}{\mathcal{H}}

\begin{document}
\title[On a Zeuthen-type problem ]{{On a Zeuthen-type problem}
}
\author[Jared Ongaro]{Jared Ongaro}
\address{School of Mathematics, University of Nairobi , 00100-Nairobi, Kenya}
\email{ongaro@uonbi.ac.ke}
\thanks{This work was part of my Doctoral studies  supported by International Science Programme, Sweden}
\begin{abstract}
We show that every degree $d$ meromorphic function on a  smooth connected projective curve $C\subset \PP^2$ of degree $d>4$ is isomorphic to a linear  projection  from a point  $p\in \mathbb {P}^2 \setminus C$ to $\PP^1$. We then pose a Zeuthen-type problem for calculating the plane Hurwitz numbers. 
\end{abstract}
\maketitle
\section{Introduction}
\label{intro}
Consider $C\subset \PP^2$, a projective plane curve  of degree $d$. An important geometric method for studying $C$,  involves  meromorphic functions arising from linear projections of $C$ from a point $p\in\PP^2$. For instance, B. Riemann established in his famous work  \cite{BR57},  that the topological structure of  a smooth curve $C\subset \PP^2$   depends  entirely on the nature  of  branch types of the branched covering $\pi_p$ arising  from a linear projection.  To construct $\pi_p$, choose a point $p\in \PP^2$ then identify $\PP^1$ with the pencil of lines passing through $p\in\PP^2$.   If $p\in \PP^2\setminus C$,  then a generic line through $p$ meets the curve $C$ in $d$  distinct points. Thus, the linear projection from a point $p\in \PP^2\setminus C$  is a finite surjective morphism 
\begin{equation}\label{linp}\pi_p:C\longrightarrow \PP^1\end{equation}
of degree $d$. The morphism $\pi_p$ is a branched covering of $\PP^1$ and the points  of $\PP^1$ where several intersection points of the corresponding line  with $C$ coincide  are the branch points of $\pi_p$.\\ 

It is  a basic problem to characterize and enumerate those meromorphic functions $f$ on $C$ which can be realized as linear projections.  First, note that in general not  all meromorphic functions on a curve $C\subset \PP^2$  can be realized as such. However,  for $d> 4$ we  have the following result which we will prove.
\begin{tm}
\label{tm:1}
Suppose that $C \subset \PP^2$ is a smooth projective plane curve of degree $d > 4$. Then any meromorphic function $f : C \longrightarrow\PP^1$  of degree $d$ can be realized as a linear projection $\pi_p:C\longrightarrow \PP^1.$ 
\end{tm}

Hurwitz numbers \cite{Hur91, OP01} count  non-isomorphic meromorphic functions  on curves with fixed genus $g$  having a fixed branched profile. On the other hand, Zeuthen numbers \cite{Ze} count nodal  plane curves of a fixed degree $d$ and geometric genus  $g$ passing through  $a$ general points  and  tangent  to $b$  general lines in $\PP^2$, where $a+b=3d+g-1$. There is  a class of  Zeuthen numbers corresponding to what we call {\em plane Hurwitz numbers}. 
Zeuthen numbers have been interpreted  by R.Vakil  in the context of stable maps as  positive degree  Gromov-Witten invariants of $\PP^2$.  In section {\ref{Zeu}} below, following \cite{Va2}, we will sketch a derivation of a class of characteristic numbers of smooth plane curves which  correspond to calculating plane Hurwitz numbers.
\subsection*{ Acknowledgements}
I am grateful to my advisor Boris Shapiro for suggesting the problem.  I also want to thank R.~ Vakil for explaining some results in  \cite{Va2}, B.~ Davison, R. ~B\o gvad, and  R. Skjelnes   for useful discussions and comments. 
\section{ General Preliminaries}
\subsection{Notation and conventions}
The base field is  $\CC$,  the field of complex numbers and we denote by $\PP^n$ the $n-$dimensional projective space over $\CC$. By a {\em variety} we  mean  a reduced algebraic projective scheme over $\CC$.  The  term  {\em curve}  means  a complete connected variety of dimension $1$.  By a {\em smooth} or {\em nonsingular}  curve  we  implicitly assume that it is irreducible.\\ 

If $\Gamma \subset \PP^n$ is a closed  subscheme,  we  write $\OO_\Gamma$  for the structure  sheaf  over $\Gamma$ and  $\II_\Gamma \subset \OO_{\PP^n}$ denotes the  ideal sheaf of $\Gamma$.  Let  $D$ be a divisor on a curve $X$, then $|D|$ is the complete linear system of $D$.  We write  $K_X$ or $K$ for the  canonical class of a smooth curve $X$ and we denote by $|K_X|$ or $|K|$ for the  complete canonical series  respectively.  Suppose that  $\FF$  is a sheaf  of  vector spaces over a projective  scheme $X$. Then we  set 
$$h^{i}(\FF):=\dim \mathbf H^i(X, \FF)  \quad \text{and} \quad \chi(\FF):=\sum_{i=0}^{\dim X} (-1)^i h^i(\FF).$$
\subsection{General Definitions}
Let  $C$  be  a nonsingular curve of genus $g$.   A surjective morphism $f: C\to \PP^1$ is called a {\em meromorphic function}. More precisely, a meromorphic function  $f$  gives a finite morphism to the complex projective line $\mathbb P^1$  whose degree $d$ by definition is  the degree of the  morphism  $f: C\longrightarrow  \mathbb P^1$.  Thus for a meromorphic function  $f$ and any fixed point $q\in \mathbb P^1$ we have the divisor $f^{-1}(q)=\mu_1p_1+\ldots + \mu_np_n,$  where $p_1, \ldots, p_n$ are pairwise distinct points on $C$ and $\mu_1, \ldots, \mu_n$ are  positive integers summing  up to $d.$  In particular,  we can assume  $\mu_1\geq \ldots\geq \mu_n$.  The partition $(\mu_1,  \ldots, \mu_n)\vdash d$ is called the {\em branch type} of $f$ at a point $q$.   For instance,  $f$ is unbranched over $q$, if the branch type equal to $(1, 1, \ldots, 1)$. The branch type for a {\em simple} branch point is $(2, 1, \ldots, 1)$. The set of all branch points is called the {\em branching locus} of $f$.   In this way, every nonconstant meromorphic function on a curve $C$ is a {\em branched covering}. The basic problem is then the classification and  enumeration of such maps $f: C\to \PP^1$ for a given $g$ and $d$ for a  prescribed branch type over each branch point  of $f$. The set of  all branch types for $f$ will be  called  {\em branch profile} of $f$.   
\subsection{Hurwitz Numbers}
\label{hu}
Branched coverings were first described  in the famous paper \cite{BR57} by Riemann who developed the idea of  representing nonsingular curves  as branched coverings of $\PP^1$ in order to study their moduli. However,  systematic investigation of branched coverings  was initiated by Hurwitz in \cite{Hur91, Hur02} more than thirty years later. 
\begin{definition}
\label{def:eqcover}
Let $f_1: C_1\to \PP^1, $ $f_2: C_2\to \PP^1$ be two branched coverings. Then $f_1$ and $f_2$  are said to be {\bf equivalent}  if there exists an  isomorphism  $h:C_1\to C_2$ such that the diagram
\begin{center}
\begin{tikzcd}[column sep=0.2cm,row sep=0.4cm]
C_1 \arrow{rr}{h} \arrow[swap]{dr}{f_1}& &C_2 \arrow{dl}{f_2}\\
& \PP^1& 
\end{tikzcd} 
\end{center}
commutes.
\end{definition}

Hurwitz observed that  if we fix the degree $d$ of the  branched coverings $f: C\to \PP^1$ and the number $w$ of  branch points  and branch profile, then equivalence classes of branched coverings form a covering space $\mathscr H_{d,g}$ (we  suppress the branch profile to avoid notational clutter) of the  configuration space of $w$ points in $\PP^1$.  These parameter spaces $\mathscr H_{d,g}$ are  called {\em Hurwitz spaces}.  The fundamental group of the configuration space of $w$ branch points in $\PP^1$ acts on the fibers of $\mathscr H_{g,d}$  and the orbits of this action are in one-one correspondence with the connected components of  $\mathscr H_{g,d}$.  A very special case is when all the branch points are simple.  In this case  there is only one orbit. This  follows that the  corresponding Hurwitz space  is an irreducible smooth algebraic variety (see\cite{Ful69}) called the  {\em small Hurwitz space} denoted   by 

\begin{equation}
\label{eq:hs}
    \begin{array}{ccc}
       \Hh_{g, d}=\left\{\mbox{ $f: C \longrightarrow \PP^1$}\ %
        \bigg |\begin{array}{c}
         \ \mbox{ {$C$ has genus $g$ and  $f$ is a branched covering}   }   \\
             \mbox{{of degree $d$  with  $w$ simple branch points}  } \\

        \end{array}
                  \right\}
                  \bigg/ \sim. \quad
    \end{array}
\end{equation}
It turns out  that  $\Hh_{d,g}$  is a covering space.   In fact it is shown in \cite{Hur91} that $\Hh_{g, d}$  comes with  a natural  finite \'{e}tale  covering   
\begin{equation}
\label{HF1}
\begin{aligned}
\Phi: \Hh_{g, d}\longrightarrow & \Sym^w \PP^1\backslash \Delta\\
(f:C\longrightarrow\PP^1)\longmapsto & \{\text{branch locus of $f$}\}
\end{aligned}
\end{equation}
where $\Sym^w\PP^1$ is the space of unordered $w-$tuples of points  in $\PP^1$ and $\Delta$ is the discriminant hypersurface corresponding to sets of cardinality  strictly less than $w$.  The Riemann-Hurwitz formula tells  us that the degree of the  branch divisor for $f: C\longrightarrow \PP^1$ in $\Hh_{g, d}$,  equals  $w=2g+2d-2$. The morphism $\Phi$ is called the {\em branching morphism} and its degree  is called the {\em simple Hurwitz number} $h_{d,g}$. Since the map $\Phi$ is finite-to-one, the  branch points can be regarded as local coordinates on  $\Hh_{g, d}$ and it follows that the dimension of the Hurwitz space is equal to $w=2g+2d-2$.   
\section{ Proof  of Theorem \ref{tm:1}}
Given  a smooth curve $C$, specifying  a meromorphic function $f : C \longrightarrow\PP^1$  of degree $d$ on $C$ corresponds to identifying an effective degree $ d$  divisor $D$  of $f$  such that the linear system $|D|$ has no base points and $\dim |D|\geq 1$.
\begin{definition}
Let $D=p_1+\ldots+p_d$ be a divisor on a smooth curve   $C$.    If $|D|$ has no base point and $\dim|D|=1$, we say that $D$ {\bf moves  in a linear pencil} $|D|$. Equivalently, we have a meromorphic function of degree $d$
$$f:C\longrightarrow \PP^1$$
such that $f^\ast \OO_{\PP^1}(1)=\LL$, where $\LL\cong \OO_C(D)$ for $\OO_C(D)$  the invertible sheaf over $C$ determined by the divisor $D$ and ${h}^0(\LL)=2$,
\end{definition}
{\bf Remark.} The assertion of Theorem \ref{tm:1} fails   if $d =3$ and $d=4$. \\
{\bf Example.} If $C \subset \PP^2$ is a smooth projective quartic, then there is a meromorphic function on $C$ of degree $4$ which is  not isomorphic to a linear  projection $\pi_p$.  Indeed let $D=p_1+\ldots +p_4$ be a divisor given by any 4 points on $C$ such that no three of them are collinear. In our case  $h^0(\LL) = 2$ by Riemann-Roch's theorem. Recall that an invertible sheaf  $\LL$ on $C$ is  base point free if $h^0(\LL)-h^0(\LL(-p))=1$ for all $p\in C$. Then  $h^0(\LL(-p))=\deg(\LL(-p))-g+1=1$  again by Riemann-Roch. So we obtain $h^0(\LL(-p))=1=h^0(\LL)-1$ and we conclude that  the linear system $|p_1 + p_2 + p_3 + p_4|$ has no base points.  Hence  the four points move  in a linear pencil but a meromorphic function specified by this divisor on a smooth quartic cannot be realized as a linear projection as this $4$ points are not in a line. 
\\

The proof of Theorem \ref{tm:1} will be  derived from the following result. 
\begin{tm}\label{tm:2}
Let $\Gamma=\{p_1,\ldots,p_d\}\subset \PP^2$,  be  any collection of $d\geq5$ distinct  points.  If $\Gamma$ fails to impose independent linear conditions on $|\OO_{\PP^2}(d -3)|$  then at least $d -1$ of the points are collinear. 
\end{tm}
To see why the proof of Theorem \ref{tm:1} follows from  that of Theorem \ref{tm:2}, recall from the introduction that to specify   a meromorphic function of degree $d$ on $C$, we specify a divisor $D$ of degree $d$ on $C$ such that the linear system $|D|$ has no base points and $\dim |D|\geq 1$, where
$$\dim|D|:=h^0(D)-1.$$
In the case  the divisor $D$ on  $C$ has a linear system as above, we  say that  $D$ moves.
\begin{definition}
The finite set $\Gamma=\{p_1, \ldots, p_d\}\subset \PP^2$ of distinct points {\em imposes linear independent conditions} on plane curves of degree $m$  if for every point $P\in \Gamma$ there exist plane curves of degree $m$ that contains $\Gamma\setminus P$ and does not contain the point $P\in \Gamma$. 
\end{definition}
Consider the subset  $\Gamma\subset \PP^2$  as a closed zero-dimensional subscheme of $\PP^2$. Then  we have  the standard exact sequence of sheaves 
\begin{equation}\label{exact}0\longrightarrow\II_{\Gamma}\otimes \OO_{\PP^2}(m)\longrightarrow \OO_{\PP^2}(m)\longrightarrow \OO_{\Gamma}(m) \longrightarrow 0 ,\end{equation}
where $\II_\Gamma\subset \OO_{\PP^2}$ is  the ideal sheaf of the zero dimensional variety  $\Gamma$. Note that $\displaystyle\OO_\Gamma(m)\cong \oplus_{i=1}^d\OO_{p_i}\cong \CC^d,$ and that surjectivity of  $$\alpha: \mathbf{H}^0(\PP^2, \OO_{\PP^2}(m))\longrightarrow \mathbf{H}^0(\Gamma, \OO_{\Gamma}(m)) $$ exactly means that there is for each $p_i$, $i=1,\ldots, d$ a plane curve  of degree $m$ that contains $\Gamma\setminus \{p_i\}$ but not $p_i$. Hence $\Gamma\subset \PP^2$ fails to impose independent conditions on curves of degree $m$ if and only if $\alpha$ is not surjective. Namely if and only if 
$$h^0(\II_\Gamma\otimes \OO_{\PP^2}(m))> h^0(\OO_{\PP^2}(m))-d=\frac{(m+1)(m+2)}{2}-d.$$
Equivalently since $\mathbf{H}^1(\PP^2, \OO_{\PP^2}(m))=0$,   $\Gamma$ fails to impose independent conditions on $|\OO_{\PP^2}(m)|$  if we have  $h^1\big( \II_{\Gamma}\otimes \OO_{\PP^2}(m)\big)> 0.$ \\

Let $D=p_1+\ldots+p_d$ be a divisor of degree $d$ on a smooth curve  $C\subset \PP^2$.  A criterion for determining  when $D$ moves  is  given by the Riemann-Roch  theorem for curves. Denote by $H$ the divisor of a general linear section.  The adjunction formula tells us that  $$K_C \sim (d-3)H.$$ 
By the  B\'ezout theorem the degree of the divisor $(d-3)H$ is  equal to $d(d-3)$.  So we obtain that
$$2g-2=(d-3)d \quad\text{or}\quad g=\frac{(d-1)(d-2)}{2}.$$
The Riemann-Roch formula implies that
$$h^0(D)=  d -g +1+h^0\big(K_C -D\big), $$
and hence $ \dim |D|\geq 1$ if and only if 
\begin{equation}\label{eq:inn}\dim |K_C -D|\geq \frac{(d-1)(d-2)}{2}-d .\end{equation}

Now the ideal sheaf $\II_{C}$ of $C$ in $\PP^2$ is isomorphic to $ \OO_{\PP^2}(-C)$,  and so  $$\mathbf{H}^0(\PP^2, \II_{C}\otimes\OO_{\PP^2}(d-3))\cong \mathbf{H}^1(\PP^2, \II_{C}\otimes\OO_{\PP^2}(d-3))=0$$
since $\mathbf{H}^0(\PP^2, \OO_{\PP^2}(-3))\cong \mathbf{H}^1(\PP^2, \OO_{\PP^2}(-3))=0.$ Twisting the  exact sequence
$$0\longrightarrow\II_{C}\longrightarrow \OO_{\PP^2}\longrightarrow \OO_{C}\longrightarrow 0 $$
by $\OO_{\PP^2}(d-3)$, we find that $\mathbf{H}^0(\PP^2, \OO_{\PP^2}(d-3))\cong \mathbf{H}^0(C, \OO_{C}(d-3)).$ Furthermore we have that  $$\mathbf{H}^0(\PP^2, \II_{\Gamma}\otimes \OO_{\PP^2}(d-3))=\ker\big (\mathbf{H}^0(\PP^2, \OO_{\PP^2}(d-3))\longrightarrow \mathbf{H}^0(\Gamma, \OO_{\Gamma}(d-3))\big).$$ 
On the other hand, $K_C \sim (d-3)H$ and $\OO_C(D)$ is the ideal of $D$ in $C$ which implies  that
$$\mathbf{H}^0(C, \OO_{C}(K_C-D))=\ker\big (\mathbf{H}^0(C, \OO_{C}(d-3))\longrightarrow \mathbf{H}^0(\Gamma, \OO_{\Gamma}(d-3))\big),$$
so we find that  $h^0( \OO_{C}(K_C-D))=h^0\big( \II_{D}\otimes \OO_{\PP^2}(d-3)\big)$.
Hence  (\ref{eq:inn})  is equivalent to the  inequality
\begin{equation}\label{impose} h^0\big( \II_{D}\otimes \OO_{\PP^2}(d-3)\big)> \frac{(d-1)(d-2)}{2}-d .\end{equation}
 In other words, the divisor $D=p_1+\ldots+p_d$ satisfies $\dim |D|\geq 1$ if and only if  the set $\Gamma=\{p_1, \ldots, p_d\}$ fails to impose  independent conditions on the canonical linear system $|K_C|$. We will now see that we may use this to derive Theorem \ref{tm:1} from Theorem \ref{tm:2}.  \\

To complete the proof  of Theorem~\ref{tm:1}, it suffices to show  that either all the $d$ points of $D$ are collinear, or  if only the $d-1$ points of $D$ lie on a line  then  the $d$-th point  is a base point of the linear system $|D|$. In the first case $D \sim H$ and we are done. In the second case, suppose that $D = p_1,\ldots, p_{d-1} + q$, where the points $p_1,\ldots, p_{d-1}$ lie on a line $\ell$ and $q\notin \ell$. We must show that  $q$ is a base point of the linear system $|D|$ or equivalently  that we have $$\dim|p_1 + \ldots + p_{d-1}|= \dim |p_1 + \ldots + p_{d-1}+q |.$$ But as the degree of the divisor $p_1 + \ldots + p_{d-1}$ is  equal to $\deg D-1$, the Riemann-Roch then implies that it is enough to show that  the following  equality: 
 \begin{equation}\label{eq:2} \dim |K_C- p_1 - \ldots - p_{d-1}-q| = \dim |K_C- p_1 - \ldots - p_{d-1}|-1\end{equation}
holds. Since $\deg C= d$, we can write the divisor cut by $C$ on $\ell$  as  $C\cdot\ell =p_1 + \ldots + p_{d-1}+b$, where $b \neq q$ because $q\notin \ell$. If a curve $C_1$ of degree $d -3$ passes through $d -1$ collinear points $p_1,  \ldots , p_{d-1}$, it must contain $\ell$ as a component. Thus, the linear system in equation (\ref{eq:2}) on left-hand side
  $$|K_C-p_1 - \ldots - p_{d-1}-q|\cong|\II_q\otimes \OO_{\PP^2}(d-4)| ,$$
whereas the linear system  on right-hand side  in  (\ref{eq:2})$$|K_C-p_1 - \ldots - p_{d-1}|\cong |\OO_{\PP^2}(d-4)|$$
which follows from the fact that $\dim |\II_q\otimes \OO_{\PP^2}(d-4)|=\dim|\OO_{\PP^2}(d-4)|-1$.
  And  this implies (\ref{eq:2}), which  completes the proof. 
\qed \\

 It is worthy to remark that if  $p_1, \ldots, p_{d-1}$ are distinct points in $\PP^2$, then they will always impose independent conditions on curves of degree $d\geq 4$. In particular,  the divisor $D=p_1+\ldots+p_{d-1}$  moves in a linear pencil if and only if the points  $p_1, \ldots, p_{d-1}$  lie on a line. It follows that for a smooth plane curve $C\subset \PP^2$ of degree $d$, there is no nonconstant meromorphic function  of degree less than $d-1$.
\section{ Proof of Theorem \ref{tm:2}} 

To shorten the proof of theorem \ref{tm:2}, we first reformulate it below in a slightly different  but equivalent form.

\begin{tm}\label{tm:2}
Let $\Gamma=\{p_0,\ldots,p_d\}\subset \PP^2$,  be  any collection of $d+1\geq5$ distinct  points.  If $\Gamma$ fails to impose independent linear conditions on $|\OO_{\PP^2}(d -2)|$  then at least $d$ of the points in $\Gamma$ are collinear. 
\end{tm}
\begin{proof}
By assumption there exists at least one point (without loss of generality) say $p_0\in  \Gamma$  such that any curve of degree $d-2$ passing through the points  in $\Gamma\setminus p_0$ also passes through $p_0$. Note that if we have a curve $C$ of degree $n\leq d-2$ that passes through $\Gamma\setminus p_0$, then it follows by assumption  that $C$ also must pass through $p_0$.\\

Let $p_0, p_1\ldots,p_j$ be the minimal number of points  in $\Gamma$ lying on a line  $\ell$ containing the point $p_0$. Rename the remaining points as $q_1, \ldots q_{d-j}$. By construction, any line through a point  $p_i\neq p_0$ and a point $q_{i}$, will not pass through $p_0$. We now construct a curve $C$ being a product of such lines. We let $\ell_i$ be the line  through $p_i$ and $q_i$ if $1\leq i\leq \min\{j, d-j\}$. For the possible remaining points,  we either let $\ell_i$ denote the lines through $p_i$ and $q_1$ (if $d-j<i\leq j$) or the line through $q_i$ and $p_1$ (if $j<i\leq d-j$). The curve
$$C=\ell_1\ldots \ell_n \qquad (\text{where}\; n=\max\{j, d-j\})$$
passes through all the points of $\Gamma\setminus p_0$, but not though $p_0$.

 If we have $2\leq j\leq d-2$ then  we get that the degree $n\leq d-2$, which is a contradiction to our assumption.\\

If we have $j=1$, then any line $\ell'$ through two points $\Gamma\setminus p_0$ would not contain $p_0$. Observe that, to cover $\Gamma\setminus p_0$,   we need at most $n\leq d/2$ lines  $\ell_1', \ldots,\ell_n' $ if $d$ is even, and  at most $n\leq (d+1)/2$ lines to cover $\Gamma\setminus p_0$, if $d$ is odd. Note that $d\geq 5$ is equivalent  to $(d+1)/2\leq d-2$, and if $d=4$ then we have  that $d/2\leq d-2$. Hence for any $d$, in our  range, we have  the curve  
$$C'=\ell_1' \ldots \ell_n'$$
of  degree $n\leq d-2$ that passes through all points of $\Gamma\setminus p_0$, but  not through $p_0$. This is impossible by assumption.\\

Finally, we are left  with the only possibility that $j> d-2$. However if $j\geq d-1$, then we have at least $j+1\leq d$ point $p_0, \ldots, p_j$ aligned on the line $\ell$. This completes the proof.
\end{proof}

\section{Plane Hurwitz numbers and Zeuthen  numbers}
\label{Zeu}
\subsection{Plane Hurwitz Numbers}
Generally in calculating  Hurwitz numbers, we make no  reference to the embedding of curves. For example, one can not expect for instance a branched covering of $\PP^1$ whose domain is  genus $2$ to be planar and smooth, since  a smooth plane curve of degree $d$, has  $g=\binom {d-1}{2}$.   Additionally, we expect that not all curves of  genus  $g=\binom {d-1}{2}$ can be embedded  in $\PP^2$ as smooth curves. For instance, among all smooth  curves of   genus $3$ (for $d=4$),  there are  hyperelliptic curves, which are not planar.\\

Fix $d>0$; the space parametrizing  all degree $d$ algebraic  curves in $\PP^2$ is  a complete   system $|\OO_{\PP^2}(d)|$, which forms a projective space $$\PP (\mathbf{H}^0(\PP^2, \OO_{\PP^2}(d)))\cong \PP^{\mathbf N},$$ where $\mathbf N=\binom{d+2}{2}-1=d(d+3)/2$. In particular,  the set of all smooth plane curves of a given  degree $d$ is an open  subset of $\PP^\mathbf N$. The   group $\PP \mathbf{GL} (3,\CC)$ of  all  projective automorphisms of $\PP^2$ acts on $\PP^\mathbf N$ in a natural way. 
 Of  particular interest is the subgroup  $\mathcal G_p\subset \PP \mathbf{GL} (3,\CC)$  fixing $p$ and preserving the pencil of lines through $p$.  
 Given a smooth curve $C\subset\PP^2$, for instance  if $p=[0:1:0]\in \PP^2\setminus C$ for some choice of  coordinate  system of $\PP^2$ an element of the  group $\mathcal G_p$ has  the  form  
$$\resizebox{.4\hsize}{!}{$g=\mbox{\small $\begin{bmatrix}g_0& 0&0\\g_1& g_2&g_3\\ 0& 0&g_0\end{bmatrix}$}\quad \text{with}\quad g_0g_2\neq 0$}.$$
The group of automorphisms $\mathcal G_p$ acts equivalently on $\PP^{\mathbf N}$  keeping the branching points of the projection $\pi_p:C\to \PP^1$  fixed. Recall from Definition \ref{def:eqcover}, that two branched coverings   $\pi_p^1:C_1\to \PP^1$ and $\pi_p^2:C_2\to \PP^1$  are  called equivalent if there exists an isomorphism  $g:C_1\to C_2$ such that $\pi_p^2\circ g=\pi_p^1.$  Then we have:
\begin{pr}
\label{pr:group}
Let $C_1, C_2\subset \PP^2$ be two smooth projective  plane curves of the same degree $d > 1$  and not passing through $p\in \PP^2$. Two projections  $\pi_{p}^1:C_1\longrightarrow \PP^1$ and $\pi_{p}^2:C_2\longrightarrow \PP^1$  are equivalent if and only if there exists an automorphism $g\in \mathcal G_p$ such that $g(C_1)=C_2.$
\end{pr}

%

\begin{proof}
Let $C_1, C_2\subset \PP^2$ be  smooth projective curves not passing through $p\in \PP^2$.  If there exists an  automorphism $g\in \mathcal G_p$ such that  $C_2=g(C_1)$,  then the morphisms $\pi_{p}$ and $\pi_{p}'$  are equivalent by an isomorphism given by $g$. For the `only if ' direction,  suppose that $\pi_{p}^1$ and $\pi_{p}^2$ are equivalent and that this equivalence is determined by an isomorphism $g: C_1\to C_2.$ For each line $\ell\ni p$ the isomorphism $g$ maps $C_1\cap \ell$ to $C_2\cap \ell;$ thus, $g$ maps hyperplane sections of $C_1$ to hyperplane sections of $C_2.$ Since both $C_1$ and $C_2$  are embedded in $\PP^2$ by complete linear system of hyperplane sections $\mathbf{H}^0(\PP^2, \OO_{C_i}(1))$, for $i=1,2$, this implies that $g$ is induced by projective automorphism $\PP \mathbf{GL} (3,\CC)$. To complete the proof,  it  only remains to check that $g\in \mathcal G_p$;  to that end, consider a generic line $\ell\ni p$; this line  intersects 
$C_i$ for $i=1,2$ at $d=\deg C_i>1$ points  and this points are mapped by $g$ to $d$ distinct points on $\ell.$ So $g(\ell)=\ell$ for the generic line and thus for any $\ell\ni p.$ If $\ell_1, \ell_2$ containing $p$ then 
$$g(p)=g(\ell_1\cap \ell_2)=g(\ell_1)\cap g(\ell_2)=\ell_1\cap \ell_2=p.$$
Hence $g\in \mathcal G_p$ as  expected  and this completes the proof. 
\end{proof}

A {\em generic projection} of smooth curve $C\subset \PP^2$ from a point $p\in \PP^2$ which is not on a bitangent line or a flex line we obtain a linear projection $\pi_p: C\to \PP^1$ with only simple branch points. This leads us to the orbit space parametrizing  all generic linear projections. Denote  this space of generic linear projections by:
\begin{equation}
\label{eq:hspm}
    \begin{array}{ccc}
       \mathcal P\Hh_{d}= \left\{\mbox{ $\pi_p : C \to \PP^1$}\ %
        \bigg |\begin{array}{c}
         \ \mbox{{$\pi_p$ is a simple linear projection  from  }  }   \\
             \mbox{{$p\in \PP^2\setminus C$  of a  smooth curve $C\subset \PP^2$}  } \\

        \end{array}
                  \right\}
                  \bigg/ \sim. \quad

    \end{array}
\end{equation}

where $\sim$ is  the  equivalence of projections from a point $p\in\PP^2$ up to  the $\mathcal G_p$- action.\\

Note that for $g=\binom{d-1}{2}$,  we have a natural inclusion  $\mathcal P\Hh_{d}\subseteq \Hh_{d,g}$  of small Hurwitz spaces for $d>1$. The information about  the  dimension of $\mathcal P\Hh_{d}$ is a direct consequence of proposition \ref{pr:group} we summarize  as follows.
\begin{co}
\label{pr:dim}
The dimension of the space  $\mathcal P\Hh_{d}$  is equal to $\mathbf N-3=\frac{d(d+3)}{2}-3$. 
\end{co}
The number of branch points  of a  generic  projection $\pi_p:C\to \PP^1$  of a smooth curve of degree $d$ from $p\in \PP^2\setminus C$ is determined by the Riemann-Hurwitz formula as $w=d(d-1)$. We refer to the number of $3$-dimensional $\mathcal G$-orbits with the same set of $w$ tangents lines  as the {\em $d$-th plane Hurwitz number}  and denote it by $\mathfrak h_{d}$.  Thus,  to compute $\mathfrak h_{d}$   as indicated in (\ref{HF1}),  we need  to calculate the degree of the branch morphism
\begin{equation}
\begin{aligned}
\mathcal P\Hh_{d}\longrightarrow & \Sym^w\PP^1\backslash \Delta,\\
\end{aligned}
\end{equation}
restricted to its image. Notice that   by Corollary \ref{pr:dim} the $\dim \mathcal P\Hh_{d} <d(d-1)$ for $d\geq 4$. Next we will give two  examples of known plane Hurwitz numbers.
\subsection*{Degree 3-plane Hurwitz Numbers}
The first nontrivial case involves  projections of smooth plane cubics.  The remark following Theorem \ref{tm:1} asserts that  if $d = 3$ not all  meromorphic function of  degree $3$ on smooth plane cubics  are realizable as projections.  However,
degree $3$ simple plane Hurwitz numbers coincides with the usually Hurwitz number. 
 Namely, over $w=6$  pairwise distinct points on the projective line $\PP^1$ there are exactly $40$ three-dimensional orbits  of smooth cubics branched over them, see \cite{Hur91}.  To see this, recall that Hurwitz numbers count branched covering up to equivalence, the equivalence of plane Hurwitz with the usual Hurwitz number is a consequence of the fact that every meromorphic function of degree $3$ on a smooth cubic is a composition of a group shift of $C$ followed by a linear  projection from $p\in \PP^2\setminus C$.   This is a well-known consequence of the fact  that any  smooth plane cubic curve  is an abelian  group. We give  the details below. 
\begin{pr}
Every meromorphic function of degree $3$ on a smooth cubic curve $C\in \PP^2$  can be represented  as a composition of a group shift on $C$ by a fixed point on $C$ with a linear projection from a point $p\in \PP^2.$  
\end{pr}
\begin{proof}
Let $C$ be a smooth projective  cubic and let  $f:C\longrightarrow \PP^1$ be a meromorphic function of degree $3$. If we  write  $f^{-1}(0)= z_1+z_2+z_3$, $f^{-1}(\infty)= p_1+p_2+p_3$ for the zero divisor and polar divisor of  $f$ respectively (where $z_i$ and $p_i$ for all $i=1,2,3$ are not necessarily distinct). The linear equivalence of divisors $f^{-1}(0) \sim f^{-1}(\infty)$ implies the equality
$$p_1+p_2+p_3= z_1+z_2+z_3$$
as divisors, where ``$+$'' denotes the addition from group law on the cubic curve. Fix  a point $P_0\in C$ such that $p_1+p_2+p_3+3P_0=0$ and define 
$$Q_i=p_i+P_0,\quad \text{and}\quad R_i=z_i+P_0\quad \text{for all $i=1,2,3.$}$$
Then we have 
\begin{align*}
\resizebox{.4\hsize}{!}{$Q_1+Q_2+Q_3=p_1+p_2+p_3+3P_0$}&=0 \\
\resizebox{.4\hsize}{!}{$R_1+R_2+R_3=z_1+z_2+z_3+3P_0$}&=0.
\end{align*}
In particular, $\{Q_1, Q_2, Q_3\}$ and   $\{R_1, R_2, R_3\}$ lie on  distinct lines in $\PP^2,$ Since otherwise these sets would be equal and so $f^{-1}(0)=f^{-1}(\infty)$, which is impossible. Denote the  lines  given  by the translates $\{Q_1, Q_2, Q_3\}$ and   $\{R_1, R_2, R_3\}$ by  $\ell_1\subset \PP^2$ and $\ell_2\subset \PP^2$ respectively. 
%
%
If  $l_1(x,y,z)$  and $l_2(x,y,z)$ are  equations for the lines $\ell_1$ and $\ell_2$, the  meromorphic  function given by composition of the group shift and projection  is the quotient $l_1/ l_2$: $f(P-P_0)=\frac{\ell_1(P)} {\ell_2(P)}\iff f(P)=\frac{\ell_1(P+P_0)} {\ell_2(P+P_0)}  $, (where $P=(x, y, x)$) after possibly multiplying with a constant using the fact that a meromorphic function without poles will be constant. 
\end{proof}
\subsection*{Degree 4-plane Hurwitz Numbers}
The case $d=4$ is more exciting. Note that  the space parametrizing projections  $\mathcal P\Hh_{4}$  has dimension $\frac{4(4+3)}{2}-3=11$. As branched coverings, this $11$-dimensional  family $\mathcal P\Hh_{4}$  admits  a natural inclusion into the small Hurwitz space $\Hh_{4, 3}$  defined in (\ref{eq:hs})  which is a smooth irreducible variety of dimension $12$. The inclusion $\mathcal P\Hh_{4}\subset \Hh_{4, 3}$ implies that the branch locus defines an hypersurface $\mathbf B\subset \Sym^{12} \PP^1$. R.~Vakil in  \cite{Va} has computed its degree to be equal  to 3762. Moreover,  he  establishes that there are  essentially 120  smooth plane quartic  branched over  admissible $12$ points in $\PP^1$. Thus,  it follows that  the plane Hurwitz number of degree $4$ is
\begin{equation}
\mathfrak h_4=120\times \frac{(3^{10}-1)}{2}.
\end{equation}   
The corresponding Hurwitz number is known to be equal to $h_{3,4}=255\times \frac{(3^{10}-1)}{2}.$

\subsection{Zeuthen numbers}
This notion of plane Hurwitz numbers has  a strong analogy to the special case of Zeuthen's classical problem which asks to calculate the number of irreducible plane  curves  of degree $d>0$ and geometric genus $g\geq 0$ passing through   $a$ general points  and $b$ tangent lines  in $\PP^2$, where $a+b=3d+g-1$.  More precisely,
assuming that the only singularities of an irreducible curve $C\subset \PP^2$ are $\delta$ nodes,  since each node reduces the freedom of the curve by $1$, we expect the set of irreducible  degree $d$ curves with $\delta$ nodes  depends on 
$$\dim|\OO_{\PP^2}(d)|-\delta=\frac{d(d+3)}{2}-\delta=3d+g-1$$
parameters.   Indeed, for  all fixed integers $d> 0$ and $g\geq 0$ as first observed by F. Severi \cite{Se} and proved by J. Harris \cite{JH86}, the Severi variety ${V}_{g, \delta}$ parametrizing irreducible plane curves of degree $d$ with $\delta$ nodes  is a quasiprojective variety of dimension $3d+g-1$.  It follows that for a fixed $d>0, g\geq 0$ the numbers $N_d(g)$ of curves  passing through  $3d+g-1$ general points is  finite and does not depend on the generic configuration of points  chosen.   
This $N_d(g)$ number is commonly referred to as {\em Severi degree} of plane curves. \\


In general, fix integers $d>0$ and $a, b, g\geq 0$.  The number of irreducible curves of geometric genus $g$ and degree $d$ passing through $ a$ general points and tangent to $b$ general lines in $\PP^2$ is finite provided $a+b=3d+g-1$. These numbers  are called {\em characteristic numbers } of plane curves and we denote them by $N_g(a, d)$.     The question of calculating {characteristic numbers} is the {\em classical problem of Zeuthen} and thus we usually refer to the numbers  $N_g(a, d)$ as {\em Zeuthen Numbers}. In \cite{Ze},  H.G. Zeuthen calculated the characteristic numbers of  smooth curves in $\PP^2$ of degree at most $4$ and  \cite{Va2} has verified Zeuthen's results using  modern results on moduli spaces of stable maps.
\subsection{Homological interpretation of Zeuthen numbers }
Let $\overline{\Mn}_{g,0}(\PP^2,d)$ be the Kontsevich moduli space  of maps  to $\PP^2$ of fixed degree $d>0$ and arithmetic genus $g\geq0$.
Consider the open substack of maps of smooth curves ${\Mn}_{g,0}(\PP^2,d)$. The closure of ${\Mn}_{g,0}(\PP^2,d)$  is a unique component of  $\overline{\Mn}_{g,0}(\PP^2,d)$  of dimension $3d+g-1$ we denote  by $\overline{\Mn}_{g,0}(\PP^2,d)^\dagger$.  The Zeuthen number $N_g(a,d)$ can be interpreted in the language  of stable maps. \\ 

Let $\alpha$ and $\beta$ denote the divisors in $\overline{\Mn}_{g,0}(\PP^2,d)^\dagger $ representing classes of a point and a line respectively.  The characteristic number $N_g(a, d)$  is given by the degree of $\alpha^a\beta^b$ and is denoted by $\alpha^a\beta^b\cap \big[\overline{\Mn}_{g,0}(\PP^2,d)^\dagger\big]$. For example, it is known there is a unique smooth cubic through $9$ general points, then we will write   
$\alpha^9\cap {\big[\overline{\Mn}_{1,0}(\PP^2,3)^\dagger\big]} =1.$ \\

The following existence result  is the key point for this interpretation. 
\begin{pr}
There  exist  two divisors $\alpha$ and $\beta$ such that the  number $N_g(a, d)$  is
$\alpha^a\beta^b\cap \big[\overline{\Mn}_{g,0}(\PP^2,d)^\dagger\big].$
\end{pr}
\begin{proof}
See \cite{Va1},  Theorem 3.15.
\end{proof}
We finish with an open problem.  As above let $\overline{\Mn}_{g,0}(\PP^2,d)^\dagger$ be  the closure of  the open substack ${\Mn}_{g,0}(\PP^2,d)$ of maps of  smooth curves of degree $d$. Among the  boundary divisors representing the closure of  loci  of maps  (see  \cite{Va1} for precise descriptions) of $\overline{\Mn}_{g,0}(\PP^2,d)^\dagger$, we have a  divisor $\mathbb I_d$ the closure of  the locus  of  degree $d:1$ maps of  smooth curves of degree $d$ into a line in $\PP^2$. Such generic maps are necessarily branched at $d(d-1)$ points  by Riemann-Hurwitz formula.  Thus the divisor  $\mathbb I_d$  enumerates a special class of Zeuthen numbers whose calculation is related to that of Hurwitz numbers. Namely, the Zeuthen numbers $\beta^{3d+g-2}[ \mathbb I_d]$  for $g=\binom{d-1}{2}$. For instance, R.~Vakil  in \cite{Va2} calculates that  $\beta^8[ \mathbb I_3]=40 \times 210$ and $\beta^{13}[\mathbb I_4]=120\cdot 2535$.  It makes sense to consider the divisor  $\mathbb I_d$  up to the $\mathcal G_p$-action. 
\section{Problem}
Consider the orbit space $\overline{\Mn}_{g,0}(\PP^2,d)^\dagger/\mathcal G_p$. Is there a natural homology class $ \beta \in\mathbf H_{2(3d+g-4)}\big (\overline{\Mn}_{g,0}(\PP^2,d\big )^\dagger/\mathcal G_p, \QQ)$ such that $\mathfrak h_d=\beta^{3d+g-5}\cap \big[\mathbb I_d/\mathcal G_p\big] $?

\end{document}